\documentclass[11pt,a4paper,reqno]{amsart}
\usepackage{graphicx} 
\usepackage{verbatim,amssymb,hyperref,mathrsfs,amsfonts}
\usepackage{dsfont}
\usepackage{ytableau}
\usepackage[all]{xy}
\usepackage{wasysym}
\usepackage[utf8]{inputenc}
\usepackage{BOONDOX-calo}
\usepackage{tikz,tikz-cd,amscd,pb-diagram}
\usetikzlibrary{arrows,decorations.markings}
\tikzstyle{dot}=[draw, fill =black, circle, inner sep=0pt, minimum size=2pt]

\numberwithin{equation}{section}

\theoremstyle{plain}
  \newtheorem{thm}{Theorem}[section]
  \newtheorem{cor}[thm]{Corollary}
  \newtheorem{lem}[thm]{Lemma}
  \newtheorem{prop}[thm]{Proposition}
\theoremstyle{definition}
  \newtheorem{Def}[thm]{Definition}
\theoremstyle{remark}
  \newtheorem{rem}[thm]{Remark}
  \newtheorem{eg}[thm]{Example}
  \newtheorem*{rem*}{Remark}
  
\numberwithin{equation}{section}

\newcommand{\bone}{\boldsymbol{1}}
\newcommand{\bl}{\boldsymbol{\lambda}}

\newcommand{\bm}{\boldsymbol{\mu}}
\newcommand{\bnu}{\boldsymbol{\nu}}
\newcommand{\UU}{\mathbf{U}}
\newcommand{\uu}{\upsilon}
\newcommand{\fs}{\mathcal{s}}
\newcommand{\Be}{\mathbf{e}}
\newcommand{\br}{\mathbf{r}}
\newcommand{\bs}{\mathbf{s}}
\newcommand{\bt}{\mathbf{t}}
\newcommand{\bu}{\mathbf{u}}
\newcommand{\bv}{\mathbf{v}}

\newcommand{\wt}{\mathsf{wt}_e}
\newcommand{\core}{\mathsf{core}_e}
\newcommand{\quot}{\mathsf{quot}_e}

\newcommand{\res}{\mathsf{res}_e}
\renewcommand{\min}{\mathrm{min}}
\renewcommand{\max}{\mathrm{max}}
\newcommand{\DDot}[1]{\mathop{\overset{#1}{\bullet}}}

\newcommand{\ee}{\varepsilon}
\newcommand{\ff}{\varphi}

\newcommand{\cs}{\mathsf{s}}
\newcommand{\B}{\mathsf{B}}
\newcommand{\BB}{\mathcal{B}}
\newcommand{\PP}{\mathcal{P}}
\newcommand{\HH}{\mathcal{H}}

\newcommand{\FF}{\mathbb{F}}
\newcommand{\EW}{\widehat{\mathbf{W}}}
\newcommand{\AW}{\mathbf{W}}
\newcommand{\sym}{\mathfrak{S}}
\newcommand{\Sp}{S}
\newcommand{\AAnobar}{\mathcal{A}_e^{\ell}}
\newcommand{\AAbar}{\overline{\mathcal{A}}_e^{\ell}}

\newcommand{\len}{\mathcal{l}}

\newcommand{\Hub}{\mathsf{hub}_e}
\newcommand{\hub}{\delta}

\newcommand{\wtH}{\mathsf{wt}_{\HH}}
\newcommand{\coreH}{\mathsf{core}_{\HH}}

\def\Z{{\mathbb Z}}

\title[Cores and weights]{Cores and weights of multipartitions \\ and
blocks of Ariki-Koike algebras}

\author{Yanbo Li}
\address[Y. Li]{School of Mathematics and Statistics, Northeastern University at Qinhuangdao, Qinhuangdao, 066004, People's Republic of China.}
\email{liyanbo707@163.com}

\author{Kai Meng Tan}
\address[K. M. Tan]{Department of Mathematics, National University of Singapore, Block S17, 10 Lower Kent Ridge Road. Singapore 119760.}
\email{tankm@nus.edu.sg}

\thanks{The first author is partially supported by Hebei Natural Science Foundation (A2021501002) and NSF of Fujian Province (No.2023J01126)}

\date{August 2024}

\subjclass[2020]{20C08, 05E10}

\begin{document}

\maketitle

\begin{abstract}
Let $e$ be an integer at least two.  We define the $e$-core and the $e$-weight of a multipartition associated with a multicharge as the $e$-core and the $e$-weight of its image under the Uglov map.
We do not place any restriction on the multicharge for these definitions. We show how these definitions lead to the definition of the $e$-core and the $e$-weight of a block of an Ariki-Koike algebra with quantum parameter $e$, and an analogue of Nakayama's `Conjecture' that classifies these blocks.  Our definition of $e$-weight of such a block coincides with that first defined by Fayers \cite{Fayers-weight}.  We further generalise the notion of a $[w:k]$-pair for Iwahori-Hecke algebra of type $A$ to the Ariki-Koike algebras, and obtain a sufficient condition for such a pair to be Scopes equivalent.
\end{abstract}

\section{Introduction}

The Ariki-Koike algebras, which are also known as cyclotomic Hecke algebras of type $G(r, 1, n)$, were introduced in \cite{AK}, \cite{Broue-Malle} and \cite{Cherednik}.
They include the Iwahori-Hecke algebras of types $A$ and $B$ as special cases, and may be considered as higher analogues of these well-studied algebras.
Besides playing an important role in modular representation
theory of finite groups of Lie type, they are related to various significant objects such as quantum groups and
rational Cherednik algebras.
Consequently, the Ariki-Koike algebras have been intensively and continuously
studied since their appearance (see \cite{Ariki-96, Ariki-01, BK-GradedDecompNos, DJM, Fayers-weight, Lyle-Mathas, Lyle-Ruff, LQ-movingvector} for examples).
The interest on these algebras is further
strengthened by a fundamental result of Brundan and Kleshchev \cite{BK-KLR}, in which an explicit isomorphism between
the blocks of Ariki-Koike algebras and cyclotomic Khovanov-Lauda-Rouquier algebras (introduced in \cite{KL,Rouquier} independently) in type
$A$ were constructed, and many profound results emerged.

As a higher analogue of the Iwahori-Hecke algebras of type $A$, one would hope that some of the results for the latter can be generalised to the Ariki-Koike algebras.  In particular, attempts (see, for example, \cite{Fayers-weight} and \cite{JL-cores}) have been made to generalise the role of $e$-cores and $e$-weights of partitions in the representation theory of the Iwahori-Hecke algebras of type $A$ to that of multipartitions for the Ariki-Koike algebras.  We make another attempt in this paper which we believe is very natural, and our generalisation leads us to partially answer a question relating to Scopes equivalence for Ariki-Koike algebras.

In this paper, for each $\ell$-partition $\bl$ associated with an $\ell$-charge $\bt$ (without any restriction on $\bt$), we define its $e$-core $\core(\bl;\bt)$ and its $e$-weight $\wt(\bl;\bt)$ to be those of the partition which is the image of such pair under the Uglov map.
The set of $\ell$-partitions associated with $\ell$-charges admits a natural right action of the extended affine Weyl group $\EW_{\ell}$ of type $A^{(1)}_{\ell-1}$.

The first main result of this paper (Theorem \ref{T:core and weight}) is that the $e$-core of $(\bl;\bt)$ is invariant under this action, while the $e$-weight of $(\bl;\bt)$ achieves the minimum in its $\EW_{\ell}$-orbit precisely when $\bt$ is ascending with its last component no more than $e$ larger than its first.

Our second main result (Corollary \ref{C:Naka}) is a version of Nakayama's `Conjecture' for the blocks of Ariki-Koike algebras: two Specht modules $\Sp^{\bl}$ and $\Sp^{\bm}$ lie in the same block of an Ariki-Koike algebra associated with the multicharge $\br$ if and only if $\core(\bl;\br) = \core(\bm;\br)$ and $\min(\wt((\bl;\br)^{\EW_{\ell}})) = \min(\wt((\bm;\br)^{\EW_{\ell}}))$.
This result leads to the definition of $e$-cores and $e$-weights of multipartitions as well as of the blocks of Ariki-Koike algebras.  We note that our definition of $e$-weights for multipartitions equals that first defined by Fayers in \cite{Fayers-weight}, even though the definitions are seemingly different at first sight.

Our third main result (Theorem \ref{T:Scopes}) relates to Scopes equivalence between two blocks of Ariki-Koike algebras (associated with the same multicharge), which is a special type of Morita equivalence.  We define $[w:k]$-pairs for two such blocks with $e$-weight $w$, and show that such pairs are Scopes equivalent when $k \geq w$.

The paper is organised as follows:  we provide the necessary background on the combinatorics of $\beta$-sets and partitions in the next section.  In Section \ref{S:core and weight}, we define the $e$-core and $e$-weight of a multipartition associated with a multicharge and prove some preliminary results, culminating with the proof of our first main theorem.  In Section \ref{S:AK}, we look at the blocks of Iwahori-Hecke algebras and prove the latter two main theorems.

Throughout, $\mathbb{Z}^+$ denotes the set of positive integers, and we fix $e,\ell \in \mathbb{Z}^+$ with $e \geq 2$.  For $a,b \in \mathbb{Z}$, we write $a \equiv_e b$ for $a \equiv b \pmod e$.  We often identify $\Z/e\Z$ with $\{0,1,\dotsc, e-1\}$ without comment.

\section{Combinatorics of $\beta$-sets and partitions}

In this section, we provide the necessary background on multi-$\beta$-sets and multipartitions, and prove some preliminary results.

\subsection{$\beta$-sets and abaci} \label{SS:beta sets}

A $\beta$-set $\B$ is a subset of $\Z$ such that both $\max(\B)$ and $\min(\Z \setminus \B)$ exist. Its charge $\fs(\B)$ is defined as
$$\fs(\B) := |\B \cap \Z_{\geq 0}| - |\Z_{<0} \setminus \B|.$$
We denote the set of all $\beta$-sets by $\BB$ and the set of all $\beta$-sets with charges $s$ by $\BB(s)$.  Clearly, $\BB= \bigcup_{s \in \Z} \BB(s)$ (disjoint union).

Any subset $S$ of $\Z$ can be represented on a number line---which is usually depicted as a horizontal line with positions labelled by $\Z$ in an ascending order going from left to right---by putting a bead at the position $a$ for each $a \in S$.  We call this the $\infty$-abacus display of $S$, and
$S$ is a $\beta$-set if and only if its $\infty$-abacus display has a least vacant position and a largest bead.

When $\B$ is a $\beta$-set, $\fs(\B)$ is the number of beads at positions $\geq 0$ minus the number of vacant positions at positions $< 0$ in its $\infty$-abacus display.
Alternatively, we can also obtain $\fs(\B)$ by repeatedly sliding beads to the vacant positions on their immediate left until all vacant positions lie on the right of all beads, and the least vacant position when this happens is $\fs(B)$.

One easily obtains the $e$-abacus display of a subset $S$ of $\Z$ from its $\infty$-abacus display by cutting up the $\infty$-abacus into sections  $\{ ae, ae+1, \dotsc, ae + e-1 \}$ ($a \in \Z$), and putting the section $\{ ae, ae+1, \dotsc, ae + e-1 \}$ directly on top of $\{ (a+1)e, (a+1)e+1, \dotsc, (a+1)e + e-1 \}$.  Thus the $e$-abacus have $e$ infinitely long vertical runners, which we label as $0, 1,\dotsc, e-1$ going from left to right, and countably infinitely many horizontal rows, labelled by $\Z$ in an ascending order going from top to bottom, and the position on row $a$ and runner $i$ equals $ae + i$.

From an $e$-abacus display of $S \subseteq \Z$, we can obtain its {\em $e$-quotient} $\quot(S) = (S_0,\dotsc, S_{e-1})$ as follows: treat each runner $i$ as an $\infty$-abacus, with the position $ae+i$ relabelled as $a$, and obtain the subset $S_i$ from the beads in runner $i$ of the $e$-abacus of $S$.  Formally,
$$
S_i := \{ a \in \Z \mid ae+i \in S \} = \{ \tfrac{x-i}{e} \mid x \in S,\ x \equiv_e i \}.
$$
It is not difficult to see that the $\quot$ is bijective as a function from the power set $\mathcal{P}(\Z)$ to $(\mathcal{P}(\Z))^e$.

Let $\B$ be a $\beta$-set, with $\quot(\B) = (\B_0,\dotsc, \B_{e-1})$.  Then it is easy to see that
\begin{equation}
\fs(\B) = \sum_{i=0}^{e-1} \fs(\B_i). \label{E:charge}
\end{equation}
Furthermore, its {\em $e$-core} $\core(\B)$ can be obtained by repeatedly sliding its beads in its $e$-abacus up their respectively runners to fill up the vacant positions above them, and its {\em $e$-weight $\wt(\B)$} is the total number of times the beads move one position up their runners.
More formally, if $\B_i = \{b_{i1} > b_{i2} > \dotsb \}$ for all $i \in \Z/e\Z$, then
\begin{align*}
\core(\B) &:= \bigcup_{i \in \Z/e\Z} \{ ae + i \mid a < \fs(\B_i) \}; \\
\wt(\B) &:= \sum_{i \in \Z/e\Z} \, \sum_{a \in \mathbb{Z}^+} (b_{ia} - \fs(\B_i)+a)
\end{align*}
(note that since each $\B_i$ is a $\beta$-set, $b_{ia} = \fs(\B_i)-a$ for all large enough $a$ so that the infinite sum in $\wt(\B)$ is actually a finite one).
We note also that
\begin{equation*}
\fs(\core(\B)) = \sum_{i \in \Z/e\Z} \fs(\B_i)= \fs(\B)  
\end{equation*}
by \eqref{E:charge}.

\begin{eg}
Let $\B = \{ 5, 4, 2, -1, -3, -4, -5,  \dotsc \}$.
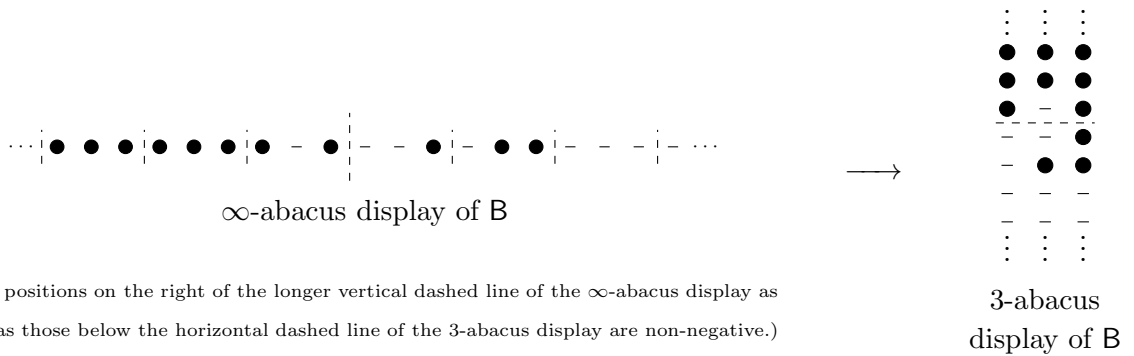
$$
\begin{matrix}
\\
\\
\\
\begin{tikzpicture}[scale=0.45]
\foreach \x in {-9,-8,-7,-6,-5,-4, -3, -1, 2, 4, 5}
{
\draw [radius=2mm, fill=black] (\x, -4) circle;
}
\foreach \x in {-2,0,1,3,6,7,8,9}
{
\draw (\x-0.15, -4) -- (\x+ 0.15, -4);
}
\foreach \x in {-10,10}
{
\node at (\x,-4) {\tiny $\dotsm$};
}

\draw [dashed] (-0.45,-5.0) -- (-0.45, -3.0);
\foreach \x in {-9,-6,-3,3,6,9}
{
\draw [dashed] (\x-0.45,-4.5) -- (\x-0.45, -3.5);
}

\end{tikzpicture}\\
\text{$\infty$-abacus display of $\B$} \\

\\
\text{\tiny (The positions on the right of the longer vertical dashed line of the $\infty$-abacus display as}\\
\text{\tiny well as those below the horizontal dashed line of the $3$-abacus display are non-negative.)}
\end{matrix}
\qquad
\xrightarrow{\phantom{mm}}
\qquad
\begin{matrix}
\begin{tikzpicture}[scale=0.5]
\foreach \x in {0,1,2}
\foreach \y in {-6,-12}
{
\node at (\x,\y) {$\vdots$};
}

\foreach \x in {0,1,2}
\foreach \y in {-7.75,-7}
{
\draw [radius=2mm, fill=black] (\x,\y) circle;
}

\draw [radius=2mm, fill=black] (0,-8.5) circle;
\draw (0.85, -8.5) -- (1.15, -8.5);
\draw [radius=2mm, fill=black] (2,-8.5) circle;

\draw (-0.15, -9.25) -- (0.15, -9.25);
\draw (0.85, -9.25) -- (1.15, -9.25);
\draw [radius=2mm, fill=black] (2,-9.25) circle;

\draw (-0.15, -10) -- (0.15, -10);
\draw [radius=2mm, fill=black] (1,-10) circle;
\draw [radius=2mm, fill=black] (2,-10) circle;

\foreach \x in {0,1,2}
\foreach \y in {-10.75,-11.5}
{
\draw (\x-0.15,\y) -- (\x+0.15,\y);
}

\draw [dashed] (-0.3,-8.875) -- (2.3,-8.875);
\end{tikzpicture} \\
\text{$3$-abacus} \\
\text{display of $\B$}
\end{matrix}
$$
Thus $\mathsf{core}_3(\B) = \{ 5,2,-1,-2,-3,\dotsc \}$ and $\mathsf{wt}_3(\B) = 2$.
\end{eg}

Given $\beta$-set $\B$ and $k \in \Z$, define
$$
\B^{+k} :=  \{ x + k \mid x \in \B \}.
$$
Then $\B^{+k}$ is a $\beta$-set with $\fs(\B^{+k}) = \fs(\B) + k$.  Furthermore, if $\quot(\B) = (\B_0,\B_1,\dotsc, \B_{e-1})$, then $\quot(\B^{+1}) = ((\B_{e-1})^{+1}, \B_0,\dotsc, \B_{e-2})$.
Consequently,
\begin{equation} \label{E:quot}
\quot(\B^{+ke}) = ((\B_0)^{+k},(\B_1)^{+k},\dotsc, (\B_{e-1})^{+k}). \end{equation}
In addition, we also have
\begin{align}
\core(\B^{+k}) &= (\core(\B))^{+k}; \label{E:corebeta} \\
\wt(\B^{+k}) &= \wt(\B).  \label{E:wtbeta}
\end{align}

We have the following easy lemma:

\begin{lem} \label{L:beta difference}
Let $\B$ and $\mathsf{C}$ be $\beta$-sets.  Then
$$|\B \setminus \mathsf{C}| - |\mathsf{C} \setminus \B| = \fs(\B) - \fs(\mathsf{C}).$$
\end{lem}

\begin{proof}
Choose $k \in \Z^+$ large enough so that $\Z_{<0} \subseteq \B^{+k} \cap C^{+k}$.  Let $\B' = \B^{+k} \cap \Z_{\geq 0}$ and $\mathsf{C}' = \mathsf{C}^{+k} \cap \Z_{\geq 0}$. Then
\begin{align*}
|\B \setminus \mathsf{C}| - |\mathsf{C} \setminus \B| &=
|\B^{+k} \setminus \mathsf{C}^{+k}| - |\mathsf{C}^{+k} \setminus \B^{+k}|
= |\B' \setminus \mathsf{C}'| - |\mathsf{C}' \setminus \B'| \\
&= |\B'| - |\mathsf{C}'|
= |\B^{+k} \cap \Z_{\geq 0}| - |\mathsf{C}^{+k} \cap \Z_{\geq 0}| \\
&= \fs(\B^{+k}) - \fs(\mathsf{C}^{+k})  \
= \fs(\B) - \fs(\mathsf{C}).
\end{align*}
\end{proof}

Following Fayers \cite{Fayers-weight}, for each $i \in \Z/e\Z$ and $\B \in \BB$, define $\hub_i(\B) = \hub_{i,e}(\B)$ by
\begin{align*}
\hub_i(\B) := |\{ x \in \B \mid x\equiv_e i,\, x-1 \notin \B \}| - |\{ x \in \B \mid x \equiv_e i-1,\, x+1 \notin \B \}|.
\end{align*}
Furthermore, call the $e$-tuple $(\hub_0(\B), \dotsc, \hub_{e-1}(\B))$ the {\em $e$-hub of $\B$}, denoted $\Hub(\B)$.

\begin{lem} \label{L:hub}
Let $\B$ be a $\beta$-set, and $i \in \Z/e\Z$.
\begin{enumerate}
\item $\hub_i(\B^{+k}) = \hub_{i-k}(\B)$.

\item If $\quot(\B) = (\B_0,\dotsc, \B_{e-1})$.  Then
$$
\hub_i(\B) =
\fs(\B_i) - \fs(\B_{i-1}) - \delta_{i0},$$ where
$\delta_{i0}$ equals $1$ when $i=0$, and $0$ otherwise.

In particular, $\Hub(\B) = \Hub(\core(\B))$.
\end{enumerate}
\end{lem}

\begin{proof}
Part (1) is clear from definition.  For part (2), we have
\begin{align*}
  \hub_i(\B) &
  = |\{ x \in \B \mid x\equiv_e i,\, x-1 \notin \B \}| - |\{ x \in \B \mid x \equiv_e i-1,\, x+1 \notin \B\}| \\
  &=|\{ a \in \B_i \mid a-\delta_{i0} \notin \B_{i-1} \} | -  |\{ a \in \B_{i-1} \mid a+\delta_{i0} \notin \B_{i} \} | \\
  &= |\B_i \setminus (\B_{i-1})^{+\delta_{i0}}| - |(\B_{i-1})^{+\delta_{i0}} \setminus \B_i| \\
  &= \fs(\B_i) - \fs((\B_{i-1})^{+\delta_{i0}}) = \fs(\B_i) - (\fs(\B_{i-1}) + \delta_{i0}),
\end{align*}
where the first equality in the last line follows from Lemma \ref{L:beta difference}.
\end{proof}

\subsection{Partitions}
A partition is a weakly decreasing infinite sequence of nonnegative integers that is eventually zero.  Given a partition $\lambda = (\lambda_1,\lambda_2,\dotsc)$, we write \begin{align*}
\len(\lambda) &:= \max \{ a \in \Z^+ : \lambda_a >0 \} ; \\
|\lambda| &:= \sum_{a=1}^{\ell(\lambda)} \lambda_a.
\end{align*}
We often identify $\lambda$ with the finite sequence $(\lambda_1,\dotsc, \lambda_{\len(\lambda)})$.
When $|\lambda| = n$, we say that $\lambda$ is a partition of $n$, and denote it as $\lambda \vdash n$.
We write $\PP(n)$ for the set of partitions of $n$ and $\PP$ for the set of all partitions.

Let $\lambda = (\lambda_1,\lambda_2,\dotsc)\in \PP$.
For each $t \in \Z$,
$$
\beta_t(\lambda) := \{ \lambda_a+t-a \mid a \in \Z^+\}
$$
is a $\beta$-set, with $\fs(\beta_t(\lambda)) = t$.  We call $\beta_t(\lambda)$ the {\em $\beta$-set of $\lambda$ with charge $t$}.

Conversely, given a $\beta$-set $\B$, say $\B = \{ b_1, b_2, \dotsc \}$ where $b_1 > b_2 > \dotsb$, let
$$\lambda_a = |\{ x \in \Z \setminus B : x < b_a \}|$$ for each $a \in \Z^+$. Then $\lambda := (\lambda_1,\lambda_2, \dotsc)$ is a partition, and it is not difficult to see that $\lambda_a = b_a + a - \fs(\B)$, so that  $\beta_{\fs(\B)}(\lambda) = \B$.

Thus, we have a bijection $\beta : \Z \times \PP \to \BB$ defined by $\beta(t, \lambda) = \beta_t(\lambda)$, and for each $t \in \mathbb{Z}$, $\beta_t(-) = \beta(t,-)$ gives a bijection between $\PP$ and $\BB(t)$.
In particular, when $\B$ is a $\beta$-set, we shall write $\beta^{-1}(\B)$ for the unique partition $\lambda$ satisfying
$$\B = \beta (\fs(\B), \lambda) = \beta_{\fs(\B)}(\lambda).$$

Note that for any partition $\lambda$ and $k \in \Z$, we have
\begin{equation}
\beta_{t+k}(\lambda) = (\beta_t(\lambda))^{+k}. \label{E:beta+}
\end{equation}
Consequently, \eqref{E:wtbeta} and \eqref{E:corebeta} imply that
\begin{align*}
\wt(\beta_{t+k}(\lambda)) = \wt((\beta_t(\lambda))^{+k}) &= \wt(\beta_t(\lambda)); \\
\core(\beta_{t+k}(\lambda)) = \core((\beta_t(\lambda))^{+k})&= (\core(\beta_t(\lambda))^{+k}).
\end{align*}
Thus we can define the $e$-weight $\wt(\lambda)$ and the $e$-core $\core(\lambda)$ of  $\lambda$ to be $\wt(\beta_t(\lambda))$ and $\beta^{-1}(\core(\beta_t(\lambda)))$ respectively for any $t \in \Z$.

We note that if $\B$ is a $\beta$-set with $\quot(\B) = (\B_0,\dotsc, \B_{e-1})$, then
\begin{align}
\quot(\core(\B)) &= (\beta_{\fs(\B_0)} (\varnothing), \dotsc, \beta_{\fs(\B_{e-1})} (\varnothing)), \label{E:quot-of-core} \\
\wt(\B) &= \sum_{i \in \Z/e\Z} |\beta^{-1}(\B_i)|. \label{E:wt-quot}
\end{align}


\subsection{Multipartitions} \label{SS:multipartitions}

A {\em multipartition} 
is an element of $\PP^m$ for some $m \in \Z^+$.
An {\em $\ell$-partition of $n$} is an element $\bl = (\lambda^{(1)}, \dotsc, \lambda^{(\ell)}) \in \PP^{\ell}$ such that $\sum_{j=1}^{\ell} |\lambda^{(j)}| = n$.  We write $\PP^{\ell}(n)$ for the set of all $\ell$-partitions of $n$.

Let $\bl \in \PP^{\ell}$.  Its {\em Young diagram} is
$$[\bl] = \{ (a,b,j) \in (\Z^+)^3 \mid a \leq \len(\lambda^{(j)}),\ b \leq \lambda^{(j)}_a,\ j \leq \ell \}.$$
Elements of $[\bl]$ are called nodes.

We usually associate $\bl$ with an {\em $\ell$-charge} $\bt \in \Z^{\ell}$, and we write $(\bl; \bt)$ for such a pair.  

Let $\bt = (t_1,\dotsc, t_{\ell}) \in \Z^{\ell}$.  The {\em $e$-residue of $(a,b,j) \in [\bl]$  relative to $\bt$}, denoted $\res^{\bt}(a,b,j)$,
is the congruence class of $b-a + t_j$ modulo $e$.
A node of $[\bl]$ with $e$-residue $i$ relative to $\bt$ is called an {\em $i$-node of $\bl$ relative of $\bt$}.
The {\em {$e$-residues of $\bl$ relative to $\bt$}}, denoted $\res^{\bt}(\bl)$, is the multiset containing precisely $\res^{\bt}(a,b,j)$ for all $(a,b,j) \in [\bl]$.

If removing an $i$-node of $\bl$ relative to $\bt$ from $[\bl]$ yields a Young diagram $[\bm]$ for some multipartition $\bm$, we call such a node {\em a removable $i$-node of $\bl$ relative to $\bt$} and {\em an addable $i$-node of $\bm$ relative to $\bt$}.
If $(a,b,j)$ is a removable $i$-node of $\bl$ relative to $\bt$, then $b = \lambda^{(j)}_a > \lambda^{(j)}_{a+1}$, and so such a node corresponds to an $x \in \beta_{t_j}(\lambda^{(j)})$ (namely $x = b-a+t_j$) such that $x \equiv_e i$ and $x -1 \notin \beta_{t_j}(\lambda^{(j)})$, and hence a bead on runner $i$ of the $e$-abacus display of $\beta_{t_j}(\lambda^{(j)})$ whose preceding position is vacant.
Similarly, an addable $i$-node of $\bl$ relative to $\bt$ corresponds to a bead on runner $i-1$ of the $e$-abacus display of $\beta_{t_j}(\lambda^{(j)})$ for some $j$ whose succeeding position is vacant.

For $\bl = (\lambda^{(1)},\dotsc, \lambda^{(\ell)}) \in \PP^{\ell}$ and $\bt = (t_1,\dotsc, t_{\ell}) \in \Z^{\ell}$, we define
$$
\hub_i^{\bt}(\bl) := \sum_{j=1}^{\ell} \delta_i(\beta_{t_j}(\lambda^{(j)}))
$$
for each $i \in \Z/e\Z$.  The {\em $e$-hub of $\bl$ relative to $\bt$}, denoted $\Hub^{\bt}(\bl)$, is then
$$
\Hub^{\bt}(\bl) := (\hub_0^{\bt}(\bl), \dotsc, \hub_{e-1}^{\bt}(\bl)).
$$

For convenience, for $\bl = (\lambda^{(1)},\dotsc, \lambda^{(\ell)}) \in \PP^{\ell}$ and $\bt = (t_1,\dotsc, t_{\ell}) \in \Z^{\ell}$, we shall use the shorthand $\beta_{\bt}(\bl)$ for
$(\beta_{t_1}(\lambda^{(1)}), \dotsc, \beta_{t_{\ell}}(\lambda^{(\ell)})) \in \BB^{\ell}$.

\subsection{Uglov's map} \label{SS:Uglov}

For each $j \in \{1, 2, \dotsc, \ell\}$, define $\uu_j: \Z \to \Z$ by
$$
\uu_j(x) = (x - \bar{x})\ell + (\ell-j)e + \bar{x}.
$$
Here $\bar{x}$ is defined by $x \equiv_e \bar{x} \in \{0,1,\dotsc,e-1\}$.   In other words, if $x = ae + b$, where $a,b \in \Z$ with $0 \leq b < e$, then $$\uu_j(x) = ae\ell + (\ell-j)e + b.$$

We record some properties of these $\uu_j$'s which can be easily verified.

\begin{lem} \label{L:Uglov map}
Let $j \in \{1,\dotsc, \ell \}$ and $x \in \Z$.  Then:
\begin{enumerate}
\item
$\uu_j(x) \equiv_e x$;
\item $\uu_j$ is injective, and $\uu_i(x) \geq 0$ if and only if $x \geq 0$;
\item $\uu_j (\Z) = \{ x \in \Z \mid \lfloor \tfrac{x}{e} \rfloor \equiv_{\ell} \ell-j \}$;
\item
$
\Z = \bigcup_{j=1}^{\ell} \uu_j(\Z)$ (disjoint union);
\item
$\uu_j(x) - e = \uu_{j+1}(x)$ if $j \in \{1,\dotsc, \ell-1\}$, and $\uu_{\ell}(x) - e = \uu_1(x-e)$.
\end{enumerate}
\end{lem}

In \cite{Uglov}, Uglov defined a map $\UU\ (= \UU_{\ell, e}) : \BB^{\ell} \to \BB$ as follows:
$$
\UU(\B^{(1)},\dotsc,\B^{(\ell)}) := \bigcup_{j=1}^{\ell}
\uu_j(\B^{(j)}).
$$
The best way to understand $\UU$ is via the abaci.
First we stack the $\infty$-abacus displays of $\B^{(j)}$'s on top of each other, with the display of $\B^{(1)}$ at the bottom and that of $\B^{(\ell)}$ at the top.
Next, we cut up this stacked $\infty$-abaci into sections with positions $\{ ae, ae+1, \dotsc, ae + e-1 \}$ ($a \in \Z$), and put the section with positions $\{ ae, ae+1, \dotsc, ae + e-1 \}$ on top of that with positions $\{ (a+1)e, (a+1)e+1, \dotsc, (a+1)e + e-1 \}$ (thus the section of the $\infty$-abacus of $\B^{(1)}$ with positions $\{ ae, ae+1, \dotsc, ae + e-1 \}$ is placed directly on top of the section of the $\infty$-abacus of $\B^{(\ell)}$ with positions $\{ (a+1)e, (a+1)e+1, \dotsc, (a+1)e + e-1 \}$), and relabel the rows of this $e$-abacus by $\Z$ in an ascending order from top to bottom, with $0$ labelling the section of the $\infty$-abacus of $\B^{(\ell)}$ with positions $\{ 0, 1, \dotsc, e-1 \}$.
This yields the $e$-abacus display of $\UU(\B^{(1)},\dotsc,\B^{(\ell)})$.

\begin{eg}
Let $\B_1 = \{-1, -3, -4, \dotsc\}$, $\B_2 = \{2,1, -1, -2, \dotsc\}$, and $\B_3 = \{2, -1, -2, -3 \dotsc\}$.
\begin{center}
\begin{tikzpicture}[scale = 0.5]
\node [left] at (-5,1) {$\B_3$};
\node [left] at (-5,2) {$\B_2$};
\node [left] at (-5,3) {$\B_1$};

\foreach \y in {1,2,3}
\foreach \x in {-4,6}
\node at (\x,\y) {$\dotsm$};

\foreach \y in {1,2,3}
\foreach \x in {-1,-3}
\draw [radius=2mm, fill=black] (\x,\y) circle;

\foreach \y in {1,2,3}
\draw (-0.15, \y) -- (0.15, \y);

\draw (-2.15, 1) -- (-1.85, 1);
\draw [radius=2mm, fill=black] (-2,2) circle;
\draw [radius=2mm, fill=black] (-2,3) circle;

\draw (0.85, 1) -- (1.15, 1);
\draw [radius=2mm, fill=black] (1,2) circle;
\draw (0.85, 3) -- (1.15, 3);

\draw (1.85, 1) -- (2.15, 1);
\draw [radius=2mm, fill=black] (2,2) circle;
\draw [radius=2mm, fill=black] (2,3) circle;

\foreach \y in {1,2,3}
\foreach \x in {3,4,5}
\draw (\x - 0.15, \y) -- (\x + 0.15, \y);

\draw [dashed] (-0.45,0) -- (-0.45, 4);
\foreach \x in {-3.45,2.55, 5.325}
\draw [dashed] (\x,0.5) -- (\x, 3.5);

\draw[->] (7.5,2) -- (9.5,2);

\foreach \x in {11,12,13}
\foreach \y in {6,-1.2}
\node at (\x,\y) {$\vdots$};

\foreach \y in {3,4,5}
\foreach \x in {11,13}
\draw [radius=2mm, fill=black] (\x,\y) circle;

\draw (11.85, 3) -- (12.15, 3);
\draw [radius=2mm, fill=black] (12,4) circle;
\draw [radius=2mm, fill=black] (12,5) circle;

\draw[dashed] (10.5,2.5) -- (13.5,2.5);

\foreach \y in {0,1,2}
\draw (10.85, \y) -- (11.15, \y);

\draw (11.85, 0) -- (12.15, 0);
\draw [radius=2mm, fill=black] (12,1) circle;
\draw (11.85, 2) -- (12.15, 2);

\draw (12.85, 0) -- (13.15, 0);
\draw [radius=2mm, fill=black] (13,1) circle;
\draw [radius=2mm, fill=black] (13,2) circle;

\node [right] at (14, 2.5) {$\UU_3(\B_1,\B_2,\B_3)$};

\end{tikzpicture}
\end{center}
Thus, $\UU_3(\B_1,\B_2,\B_3) = \{5,4,2,-1,-3,-4,-5,\dotsc\}$.
\end{eg}

From the abacus description of $\UU$, it is not difficult to see that the map $\UU$ gives a bijection between the set of $\ell$-tuples of $\beta$-sets and the set of $\beta$-sets.
Furthermore, the following lemma holds:

\begin{lem} \label{L:}
Let $\B^{(1)}, \dotsc, \B^{(\ell)}$ be $\beta$-sets.
\begin{enumerate}
\item For each $k \in \Z$, we have
$$\UU((\B^{(1)})^{+ke},\dotsc, (\B^{(\ell)})^{+ke}) = (\UU(\B^{(1)},\dotsc, \B^{(\ell)}))^{+ke\ell }.$$

\item Let $\quot(\UU(\B^{(1)},\dotsc,\B^{(\ell)})) = (\B_0,\dotsc, \B_{e-1})$ and $\quot(\B^{(j)}) = (\B^{(j)}_0,\dotsc, \B^{(j)}_{e-1})$ for each $j\in \{ 1,\dotsc, \ell\}$.  Then
    $$\fs(\B_i) = \sum_{j=1}^{\ell} \fs(\B^{(j)}_i)$$ for all $i \in \Z/e\Z$.
    In particular,
    $$\fs(\UU(\B^{(1)},\dotsc,\B^{(\ell)})) = \sum_{i \in \Z/e\Z}\, \sum_{j=1}^{\ell} \fs(\B^{(j)}_i) = \sum_{j=1}^{\ell} \fs(\B^{(j)}).$$

\item For each $i \in \Z/e\Z$, we have
$$
\hub_i(\UU(\B^{(1)},\dotsc, \B^{(\ell)})) = \sum_{j=1}^{\ell} \hub_i(\B^{(j)})  + (\ell -1) \delta_{i0}
$$
(where $\delta_{i0}$ equals $1$ when $i=0$, and $0$ otherwise, as before).
\end{enumerate}
\end{lem}

\begin{proof} \hfill
\begin{enumerate}
\item
This follows since for each $j \in \{1,\dotsc, \ell\}$, $\uu_j (x+ke) = \uu_j(x) + ke\ell$ for all $x \in \Z$, so that $\uu_j((\B^{(j)})^{+ke}) = (\uu_j(\B^{(j)}))^{+ke \ell }$.

\item
Assume first that $\Z_{< 0} \subseteq \B^{(j)}$ for all $j \in \{ 1,\dotsc, \ell \}$.
Then $\Z_{<0} \subseteq \B^{(j)}_i$ for all $i \in \Z/e\Z$ and $j \in \{ 1,\dotsc, \ell \}$.
We have
\begin{align*}
\B_i &= \left\{ \tfrac{x-i}{e} \mid x \in \UU(\B^{(1)},\dotsc, \B^{(\ell)}),\ x\equiv_e i \right\} \\
&= \left\{ \tfrac{x-i}{e} \mid x \in \bigcup_{j=1}^{\ell} \uu_j(\B^{(j)}),\ x \equiv_e i \right\} \\
&= \bigcup_{j=1}^{\ell} \left\{ \tfrac{\uu_j(x_j)-i}{e} \mid x_j \in \B^{(j)},\ x_j \equiv_e i \right\} \quad \text{(disjoint union)}
\end{align*}
by Lemma \ref{L:Uglov map}(1,4). Consequently, since $\Z_{< 0} \subseteq \B^{(j)}$ for all $j$, 
we see that $\Z_{<0} \subseteq \B_i$.
Thus
\begin{align*}
\fs(\B_i)
= |\B_i \cap \Z_{\geq 0}|
&= | \bigcup_{j=1}^{\ell} \{ \tfrac{\uu_j(x_j)-i}{e} \mid x_j \in \B^{(j)},\ x_j \equiv_e i,\ \tfrac{\uu_j(x_j)-i}{e} \geq 0 \}| \\
&= \sum_{j=1}^{\ell} |\{ x \in \B^{(j)} \mid x \equiv_e i,\ x \geq 0 \}| \\
&= \sum_{j=1}^{\ell} |\{ y \in \B^{(j)}_i \mid y \geq 0 \}|
= \sum_{j=1}^{\ell} |\B^{(j)}_i \cap \Z_{\geq 0}|
= \sum_{j=1}^{\ell} \fs(\B^{(j)}_i)
\end{align*}
as required.

For general $\B^{(j)}$'s,
let $k \in \Z^+$ be large enough so that $\mathsf{C}^{(j)} := (\B^{(j)})^{+ke} \supseteq \Z_{<0}$ for each $j \in \{1,\dotsc, \ell\}$.
Then  $\quot(\mathsf{C}^{(j)}) = ((\B^{(j)}_0)^{+k},\dotsc, (\B^{(j)}_{e-1})^{+k})$ by \eqref{E:quot}.
Furthermore, by part (1), we have
$$
\UU(\mathsf{C}^{(1)},\dotsc, \mathsf{C}^{(\ell)}) = \UU((\B^{(1)})^{+ke},\dotsc, (\B^{(\ell)})^{+ke}) = (\UU(\B^{(1)}, \dotsc, \B^{(\ell)}))^{+ke\ell},
$$
so that
$$
\quot(\UU(\mathsf{C}^{(1)},\dotsc, \mathsf{C}^{(\ell)})) = \quot ((\UU(\B^{(1)}, \dotsc, \B^{(\ell)}))^{+ke\ell}) = (\B_0^{+k\ell},\dotsc, \B_{e-1}^{+k\ell})
$$
by \eqref{E:quot}.
Applying what we have shown in the last paragraph to the $i$-th component of the $e$-quotient of $\UU( \mathsf{C}^{(1)},\dotsc, \mathsf{C}^{(\ell)})$, we get
$$
\fs((\B_i)^{+k\ell}) = \sum_{j=1}^{\ell} \fs((\B^{(j)}_i)^{+k}),
$$
which yields $\fs(\B_i) = \sum_{j=1}^{\ell} \fs(\B^{(j)}_i)$ as desired.

The final assertion follows since $\fs(\UU(\B^{(1)},\dotsc, \B^{(\ell)})) = \sum_{i \in \Z/e\Z} \fs(\B_i)$ and $\fs(\B^{(j)}) = \sum_{i \in \Z/e\Z} \fs(\B^{(j)}_i)$ by \eqref{E:charge}.

\item Keep the notations used in part (2). Then
\begin{alignat*}{2}
\hub_i(\UU(\B^{(1)},\dotsc, \B^{(\ell)}))  &= \fs(\B_i) - \fs(\B_{i-1}) - \delta_{i0} &\qquad& 
\\
&= \sum_{j=1}^{\ell} \left(\fs(\B^{(j)}_i) - \fs(\B^{(j)}_{i-1})\right) - \delta_{i0} &\qquad&
\\
&= \sum_{j=1}^{\ell} \hub_i(\B^{(j)}) + (\ell-1)\delta_{i0}
\end{alignat*}
by Lemma \ref{L:hub}(2) and part (2).
\end{enumerate}
\end{proof}

\subsection{Place permutation action of the symmetric group}

Let $m \in \Z^+$.  The symmetric group on $m$ letters is denoted by $\mathfrak{S}_m$.  We view $\mathfrak{S}_m$ as a group of functions acting on $\{1,2,\dotsc,m\}$ so that we compose from right to left; for example, $(1,2)(2,3) = (1,2,3)$.  This is a Coxeter group of type $A_{m-1}$, with Coxeter generators $\cs_a = (a,a+1)$ for $1 \leq a \leq m-1$.

Given any nonempty set $X$, $\sym_{m}$ has a natural right place permutation action on $X^{m}$ via
$$
(x_1,\dotsc, x_{m})^{\sigma} = (x_{\sigma(1)},\dotsc, x_{\sigma(m)})
$$
for $\sigma \in \sym_m$ and $x_1,\dotsc, x_m \in X$.

We record a simple lemma.

\begin{lem} \label{L:sym}
Let $X$ be a totally ordered set, and let $(x_1,\dotsc, x_{m}) \in X^{m}$ be an ascending sequence.  If $(x_1,\dotsc, x_{m})^{\sigma}$ is also ascending for some $\sigma\in \mathfrak{S}_{m}$, then $(x_1,\dotsc, x_{\ell})^{\sigma} = (x_1,\dotsc, x_{\ell})$ and
$\sigma \in \left< \cs_a \mid x_a = x_{a+1} \right>$.
\end{lem}

\begin{proof}
Since $(x_1,\dotsc, x_{m})^{\sigma}$ is a rearrangement of $(x_1,\dotsc, x_{m})$, if both sequences are ascending, then they must be equal.  Consequently, $\sigma$ only permutes within each subset of $x_a$'s which are equal, and the lemma follows.
\end{proof}

\subsection{The extended affine Weyl group $\EW_m$} \label{SS:Weyl group action}

Let $\{ \Be_1, \dotsc, \Be_m\}$ denote the standard $\Z$-basis for the free (left) $\Z$-module $\Z^m$.
The right place permutation action of $\mathfrak{S}_m$ on $\Z^m$ induces the semidirect product $\EW_m = \Z^m \rtimes \mathfrak{S}_m$, which is an extended affine Weyl group of type $A^{(1)}_{m-1}$.
It contains the affine Weyl group $\AW_m$ of type $\widetilde{A}_{m-1}$ generated by $\{ \cs_0,\cs_1, \dotsc, \cs_{m-1}\}$, where $\cs_a = (a,a+1) \in \sym_m$ for $1 \leq a \leq m-1$ as before, and $\cs_0 = (\Be_1 - \Be_m)(1,m)$.

We shall use two actions of the extended affine Weyl group in this paper.  The first is a right action of $\EW_{\ell}$ on $\Z^{\ell}$ and $\PP^{\ell}$  via
\begin{align*}
\bt^{\sigma \bu} = \bt^{\sigma} + e\bu \quad \text {and} \quad
\bl^{\sigma \bu} = \bl^{\sigma}
\end{align*}
for $\bl \in \PP^{\ell}$, $\bt, \bu \in \Z^{\ell}$ and $\sigma \in \sym_{\ell}$, which further induces a right action on $\PP^{\ell}\times \Z^{\ell}$.
We write $x^{\EW_{\ell}}$ for the orbit of $x$ (for $x$ in $\Z^{\ell}$ or $\PP^{\ell}$ or $\PP^{\ell}\times \Z^{\ell}$), and note that $\res^{\bt}(\bl)$ and $\Hub^{\bt}(\bl)$ are invariant under this action; i.e.\
\begin{equation}
\res^{\bt^w}(\bl^w) = \res^{\bt}(\bl) \quad \text{and} \quad \Hub^{\bt^w}(\bl^w) = \Hub^{\bt}(\bl) \label{E:res-hub}
\end{equation}
for all $w \in \EW_{\ell}$.

The second is a left action $\DDot{k}$ of $\EW_e$ on $\Z$ (for $k \in \Z^+$) defined by
\begin{alignat*}{2}
\cs_i \DDot{k} x &=
\begin{cases}
x-1, &\text{if } x \equiv_e i \\
x+1, &\text{if } x \equiv_e i-1 \\
x, &\text{otherwise}
\end{cases}
&\qquad & (i \in \{ 1,\dotsc, e-1 \}); \\
\Be_j \DDot{k} x &=
\begin{cases}
x+k e, &\text{if } x \equiv_e j-1 \\
x, &\text{otherwise}
\end{cases}
&& (j \in \{ 1,\dotsc, e\}).
\end{alignat*}
This action is easily described using the $e$-abacus:
$\cs_i$ ($1 \leq i \leq e-1$) swaps the runners $i-1$ and $i$, while $\Be_j$ ($1 \leq j \leq e$) moves each position on runner $j-1$ down $k$ positions.
We note that $\cs_i \DDot{k} - = \cs_i \DDot{1} -$ for all $i \in \{1,\dotsc, e-1\}$ and $k \in \Z^+$, while $\cs_0 \DDot{1} - $ has a similar description as $\cs_i \DDot{1} -$ for $i \in \{1,\dotsc, e-1\}$, i.e.
$$
\cs_0 \DDot{1} x =
\begin{cases}
x-1, &\text{if } x \equiv_e 0; \\
x+1, &\text{if } x \equiv_e e-1; \\
x, &\text{otherwise}.
\end{cases}
$$

We extend this left action $\DDot{k}$ to the powerset $\mathcal{P}(\Z)$ of $\Z$, i.e.\
$g \DDot{k} X = \{ g \DDot{k} x \mid x \in X \}$ for $X \subseteq \Z$.  The set $\BB$ of $\beta$-sets is invariant under this action.
Furthermore, for each $i \in \Z/e\Z$, $\cs_i \DDot{k} -$ leaves the charge and the $e$-weight of each $\beta$-set unchanged, and commutes with taking $e$-core, i.e.\
\begin{align}
\fs(\cs_i \DDot{k} \B) &= \fs(\B); \notag \\
\wt(\cs_i \DDot{k} \B) &= \wt(\B); \label{E:cs_i-wt} \\
\core(\cs_i \DDot{k} \B) &= \cs_i \DDot{k} \core(\B). \label{E:cs_i-core}
\end{align}
In addition, we note also that
\begin{equation}
\cs_i \DDot{k} (\B^{+me}) = (\cs_i \DDot{k} \B)^{+me} \label{E:cs+}
\end{equation}
for all $\B \in \BB$ and $m \in \Z$.

We further extend this left action to $(\mathcal{P}(\Z))^{m}$ for any $m \in \mathbb{Z}^+$ by acting on each and every component.
From the abacus description of $\DDot{k}$, it is easy to see that for $\beta$-sets $\B^{(1)},\dotsc, \B^{(\ell)}$, we have
\begin{equation}
\UU(\cs_i \DDot{1} (\B^{(1)},\dotsc, \B^{(\ell)})) = \cs_i \DDot{\ell} \UU(\B^{(1)},\dotsc, \B^{(\ell)}) \label{E:cs_i-U}
\end{equation}
for all $i \in \Z/e\Z$.

For $\bl = (\lambda^{(1)},\dotsc, \lambda^{(\ell)}) \in \PP^{\ell}$, $\bt = (t_1,\dotsc, t_{\ell}) \in \Z^{\ell}$ and $i \in \Z/e\Z$, we have
\begin{align*}
\cs_i \DDot{1} \beta_{\bt}(\bl) = (\cs_i \DDot{1} \beta_{t_1}(\lambda^{(1)}), \dotsc, \cs_i \DDot{1} \beta_{t_{\ell}}(\lambda^{(\ell)})).
\end{align*}
Since $\fs(\cs_i \DDot{1} \beta_{t_j}(\lambda^{(j)})) = \fs(\beta_{t_j}(\lambda^{(j)})) = t_j$ for all $j \in \{1,\dotsc, \ell\}$, we have
$\cs_i \DDot{1} \beta_{t_j}(\lambda^{(j)}) = \beta_{t_j}(\mu^{(j)})$ for a unique $\mu^{(j)} \in \PP^{\ell}$.
Writing $\cs_i \DDot{\bt} \bl$ for $\bm := (\mu^{(1)},\dotsc, \mu^{(\ell)})$, we get a left action $\DDot{\bt}$ of $\AW_e$  on $\PP^{\ell}$, satisfying
$$
\beta_{\bt}(\cs_i \DDot{\bt} \bl) = \cs_i \DDot{1} \beta_{\bt}(\bl),$$
which further induces a left action $\DDot{}$ of $\AW_e$ on $\PP^{\ell} \times \Z^{\ell}$ via $\cs_i \DDot{} (\bl;\bt) = (\cs_i \DDot{\bt} \bl; \bt)$.
Note that $\cs_i \DDot{\bt} \bl$ is the multipartition obtained from $\bl$ by removing all its removable $i$-nodes and adding all its addable $i$-nodes, both relative to $\bt$.

Now, for $\sigma \in \sym_m$ and $\bu = (u_1,\dotsc, u_{\ell}) \in \Z^{\ell}$, we have, by \eqref{E:beta+} and \eqref{E:cs+},
\begin{align*}
\cs_i \DDot{1} \beta_{\bt^{\sigma \bu}}(\bl^{\sigma})
&= (\cs_i \DDot{1} \beta_{t_{\sigma(1)} + eu_1} (\lambda^{\sigma(1)}), \dotsc, \cs_i \DDot{1} \beta_{t_{\sigma(\ell)} + eu_{\ell}} (\lambda^{\sigma(\ell)})) \\
&= ((\cs_i \DDot{1} \beta_{t_{\sigma(1)}} (\lambda^{\sigma(1)}))^{+eu_1}, \dotsc, (\cs_i \DDot{1} \beta_{t_{\sigma(\ell)}} (\lambda^{\sigma(\ell)}))^{+eu_{\ell}}) \\
&= ((\beta_{t_{\sigma(1)}}(\mu^{(\sigma(1))}))^{+eu_1}, \dotsc, (\beta_{t_{\sigma(\ell)}}(\mu^{(\sigma(\ell))}))^{+eu_{\ell}}) \\
&= (\beta_{t_{\sigma(1)}+eu_1}(\mu^{(\sigma(1))}), \dotsc, \beta_{t_{\sigma(\ell)}+eu_{\ell}}(\mu^{(\sigma(\ell))})) = \beta_{\bt^{\sigma\bu}}(\bm^{\sigma}).
\end{align*}
Thus
$$
\cs_i \DDot{} ((\bl;\bt)^{\sigma\bu}) = (\cs_i \DDot{\bt^{\sigma\bu}} \bl^{\sigma}; \bt^{\sigma\bu}) = (\bm^{\sigma}; \bt^{\sigma\bu})
=(\bm;\bt)^{\sigma\bu} = (\cs_i \DDot{\bt}\bl; \bt)^{\sigma\bu} = (\cs_i \DDot{} (\bl;\bt))^{\sigma\bu};
$$
i.e.\ the left action $\DDot{}$ of $\AW_e$ on $\PP^{\ell} \times \Z^{\ell}$ commutes with the right action of $\EW_{\ell}$.

\section{Cores and weights of multipartitions} \label{S:core and weight}

In this section, we define the $e$-weight and the $e$-core of a multipartition associated to a multicharge, and study their properties. In particular, we investigate how these change under the right action of the extended affine Weyl group $\EW_{\ell}$.

\begin{Def}
Let $\bl = (\lambda^{(1)},\dotsc, \lambda^{(\ell)}) \in \PP^{\ell}$ and $\bt = (t_1,\dotsc, t_{\ell}) \in \Z^{\ell}$.  Recall that $\beta_{\bt}(\bl) = (\beta_{t_1}(\lambda^{(1)}),\dotsc, \beta_{t_{\ell}}(\lambda^{(\ell)}))$.
\begin{enumerate}

\item We write $\UU(\bl;\bt)$ for $\beta^{-1}(\UU(\beta_{\bt}(\bl)))$. (Thus, $\UU(\bl;\bt) \in \PP$ and $\beta_t(\UU(\bl;\bt)) = \UU(\beta_{\bt}(\bl))$, where $t = \sum_{j=1}^{\ell} t_j$.)

\item We write $\core(\bl;\bt) \in \PP$ and $\wt(\bl;\bt) \in \mathbb{Z}$ for the $e$-core and the $e$-weight of $\UU(\bl;\bt)$ respectively; i.e.\
\begin{alignat*}{2}
\core(\bl;\bt) &:= \core(\UU(\bl;\bt)) &&= \beta^{-1}(\core(\UU(\beta_{\bt}(\bl)))), \\
\wt(\bl;\bt) &:= \wt(\UU(\bl;\bt)) &&= \wt(\UU(\beta_{\bt}(\bl))).
\end{alignat*}
\end{enumerate}
\end{Def}

Recall the right action of the extended affine Weyl group $\EW_{\ell}$ on $\PP^{\ell} \times \Z^{\ell}$ via
$$
(\bl;\bt)^{\sigma\bu} = (\bl^{\sigma}; \bt^{\sigma} +e\bu)
$$
as detailed in Subsection \ref{SS:Weyl group action}.


\begin{prop} \label{P:Uglov-partition}
Assume $\ell \geq 2$, and let $\bl \in \PP^{\ell}$ and $\bt = (t_1,\dotsc, t_{\ell}) \in \Z^{\ell}$.
\begin{enumerate}
\item Let $\rho = (1,2,\dotsc, \ell) \in \mathfrak{S}_{\ell}$.  Then
$$
\UU((\bl; \bt)^{\rho\Be_{\ell}})  = \UU(\bl;\bt).
$$

\item For all $j \in \{1,\dotsc, {\ell}\}$ and $k \in \{1,\dotsc, {\ell}-1 \}$, we have
$$\core((\bl; \bt)^{\Be_j}) = \core(\bl;\bt) = \core((\bl;\bt)^{\cs_k}).$$

\item For all $k \in \{1,\dotsc, {\ell}-1\}$, we have
$$\wt((\bl; \bt)^{\cs_k}) = \wt(\bl;\bt) + t_{k+1} - t_k.$$

\item $$\wt((\bl; \bt)^{\cs_0}) = \wt(\UU(\bl;\bt)) + t_1 - t_{\ell} + e.$$

\item Let $\bone := \sum_{j=1}^{\ell} \mathbf{e}_j$.  Then
$$
\wt(\bl;\bt) = \wt(\bl;\bt + \bone).
$$
\end{enumerate}
\end{prop}

\begin{proof} \hfill
\begin{enumerate}
\item
Let $t = \sum_{j=1}^{\ell} t_j$.
We have
\begin{align*}
\UU(\beta_{\bt^{\rho} + e\Be_{\ell}}(\bl^{\rho}))
&= \UU(\beta_{(t_2,\dotsc, t_{\ell}, t_1+e)} (\lambda^{(2)},\dotsc, \lambda^{({\ell})}, \lambda^{(1)})) \\
&= \bigcup_{j=1}^{{\ell}-1} \uu_j(\beta_{t_{j+1}}(\lambda^{(j+1)})) \cup
\uu_{\ell}(\beta_{t_1+e}(\lambda^{(1)}))  \\
&= \bigcup_{j=1}^{{\ell}-1} \{\uu_j(x) : x \in \beta_{t_{j+1}}(\lambda^{(j+1)}) \} \cup \{
\uu_{\ell}(x) : x \in \beta_{t_1+e}(\lambda^{(1)}) \} \\
&= \bigcup_{j=1}^{{\ell}-1} \{ \uu_{j+1}(x) +e : x \in \beta_{t_{j+1}}(\lambda^{(j+1)}) \}
\cup \{
\uu_1(x-e) +e : x-e \in \beta_{t_1}(\lambda^{(1)}) \} \\
&= \left(\bigcup_{j=1}^{\ell} \uu_j(\beta_{t_j}(\lambda^{(j)})) \right)^{+e}
= (\UU(\beta_{\bt}(\bl)))^{+e} = (\beta_t(\UU(\bl;\bt)))^{+e} = \beta_{t+e}(\UU(\bl;\bt)),
\end{align*}
where the fourth and the last equalities follow from Lemma \ref{L:Uglov map}(5) and \eqref{E:beta+} respectively.
Thus
$$ \UU((\bl;\bt)^{\rho\Be_{\ell}}) = \UU(\bl^{\rho}; \bt^{\rho} + e\Be_{\ell}) =
\beta^{-1}(\UU(\beta_{\bt^{\rho} +e\Be_{\ell}}(\bl^{\rho}))) = \beta^{-1}(\beta_{t+e}(\UU(\bl;\bt))) = \UU(\bl; \bt).$$

\item
Let
\begin{align*}
\quot(\UU(\beta_{\bt}(\bl))) &= (\B_0,\dotsc, \B_{e-1}); \\
\quot(\UU(\beta_{\bt^{\Be_j}}(\bl))) &= (\B'_0,\dotsc, \B'_{e-1}); \\
\quot(\UU(\beta_{\bt^{\cs_k}}(\bl^{\cs_k}))) &= (\B''_0,\dotsc, \B''_{e-1}).
\end{align*}
The $e$-abacus display of $\UU(\beta_{\bt^{\Be_j}}(\bl))$ is obtained from that of $\UU(\beta_{\bt}(\bl))$ by moving its row $r$ to row $r+\ell$ for all $r \equiv_{\ell} \ell - j+1$.
Thus $\fs(\B'_i) = \fs(\B_i) + 1$ for all $i \in \Z/e\Z$.
Consequently,
\begin{align*}
\quot(\core(\UU(\beta_{\bt^{\Be_j}}(\bl))))
&= (\beta_{\fs(\B'_0)}(\varnothing), \dotsc, \beta_{\fs(\B'_{e-1})}(\varnothing)) \\
&= (\beta_{\fs(\B_0)+1}(\varnothing), \dotsc, \beta_{\fs(\B_{e-1})+1}(\varnothing)) \\
&= ((\beta_{\fs(\B_0)}(\varnothing))^{+1}, \dotsc, (\beta_{\fs(\B_{e-1})}(\varnothing))^{+1}) \\
&= \quot((\core(\UU(\beta_{\bt}(\bl))))^{+e})
\end{align*}
by \eqref{E:quot-of-core}, \eqref{E:beta+} and \eqref{E:quot}. Hence, $\core(\UU(\beta_{\bt^{\Be_j}}(\bl))) = (\core(\UU(\beta_{\bt}(\bl))))^{+e}$, and so
\begin{align*}
\core((\bl;\bt)^{\Be_j})
= \beta^{-1}(\core(\UU(\beta_{\bt^{\Be_j}}(\bl))))
&= \beta^{-1}((\core (\UU(\beta_{\bt}(\bl))))^{+e}) \\
&= \beta^{-1}(\core(\UU(\beta_{\bt}(\bl)))) = \core(\bl;\bt).
\end{align*}
On the other hand, the $e$-abacus display of $\UU(\beta_{\bt^{\cs_k}}(\bl^{\cs_k}))$ is obtained from that of $\UU(\beta_{\bt}(\bl))$ by swapping its row $r$ with row $r-1$ for all $r \equiv_{\ell} \ell - k$.
Thus, $\fs(\B''_i) = \fs(\B_i)$ for all $i \in \Z/e\Z$, and hence $\core(\UU(\beta_{\bt^{\cs_k}}(\bl^{\cs_k}))) = \core(\UU(\beta_{\bt}(\bl)))$, so that
$$
\core((\bl;\bt)^{\cs_k}) = \beta^{-1}(\core(\UU(\beta_{\bt^{\cs_k}}(\bl^{\cs_k})))) =
\beta^{-1}(\core(\UU(\beta_{\bt}(\bl)))) =
\core(\bl;\bt).
$$

\item
Note that the $e$-abacus display of $\UU(\beta_{\bt^{\cs_k}}(\bl^{\cs_k}))$ is obtained from that of $\UU(\beta_{\bt}(\bl))$ by swapping its row $r$ with row $r-1$ for all $r \equiv_{\ell} \ell -k$.
As such, going from $\UU(\beta_{\bt}(\bl))$ to $\UU(\beta_{\bt^{\cs_k}}(\bl^{\cs_k}))$,
the $e$-weight is increased by $1$ for each occurrence of a vacant position in row $r$ of $\UU(\beta_{\bt}(\bl))$ with a bead directly above it (in row $r-1$), and is decreased by $1$ for each occurrence of a bead in row $r$ of $\UU(\beta_{\bt}(\bl))$ with a vacant position directly above it, for each $r \equiv_{\ell} \ell -k$.
Now each position in row $r$ with $r \equiv_{\ell} \ell - k$ is of the form $\uu_k(x)$ for some $x \in \Z$, and the position directly above it is then $\uu_k(x) -e = \uu_{k+1}(x)$ by Lemma \ref{L:Uglov map}(5).
Thus, the $e$-weight is increased by $1$ for each $x \in \beta_{t_{k+1}}(\lambda^{(k+1)}) \setminus \beta_{t_{k}}(\lambda^{(k)})$ and is decreased by $1$ for each $x \in \beta_{t_{k}}(\lambda^{(k)}) \setminus \beta_{t_{k+1}}(\lambda^{(k+1)})$.  Consequently
\begin{align*}
\wt(\UU(\beta_{\bt^{\cs_k}}(\bl^{\cs_k})))
&=
\wt(\UU(\beta_{\bt}(\bl)))
+ |\beta_{t_{k+1}}(\lambda^{(k+1)}) \setminus \beta_{t_k}(\lambda^{(k)})|
- |\beta_{t_k}(\lambda^{(k)}) \setminus \beta_{t_{k+1}}(\lambda^{(k+1)})| \\
&= \wt(\UU(\beta_{\bt}(\bl))) + t_{k+1} - t_{k}
\end{align*}
by Lemma \ref{L:beta difference}.

\item
Note first that in $\EW_{\ell}$, we have
$
\cs_0 \rho\Be_{\ell} =
\rho\Be_{\ell}\cs_{\ell-1}.
$
Thus
\begin{alignat*}{2}
\wt(\UU((\bl;\bt)^{\cs_0}))
&= \wt(\UU((\bl;\bt)^{\cs_0 \rho\Be_{\ell}})) &\qquad &(\text{by (1)}) \\
&= \wt(\UU((\bl;\bt)^{\rho\Be_{\ell}\cs_{\ell-1}})) &&\\
&= \wt(\UU((\bl;\bt)^{\rho\Be_{\ell}})) + t_1 + e - t_{\ell} & &(\text{by (3) since } \bt^{\rho\Be_{\ell}}=(t_2, \dots, t_{\ell}, t_1+e))\\
&= \wt(\UU((\bl;\bt))) + t_1 + e - t_{\ell} &\qquad &(\text{by (1)}).
\end{alignat*}

\item Let $\quot(\UU(\beta_{\bt}(\bl))) = (\B_0,\dotsc, \B_{e-1})$.  Then $\quot(\UU(\beta_{\bt + \bone}(\bl))) = ((\B_{e-1})^{+\ell}, \B_0,\dotsc, \B_{e-2})$.
    Consequently,
    \begin{align*}
    \wt(\UU(\beta_{\bt+\bone}(\bl))) &= |\beta^{-1}((\B_{e-1})^{+\ell})| + |\beta^{-1}(\B_0)| + \dotsb + |\beta^{-1}(\B_{e-2})| \\
    &= \sum_{i\in \Z/e\Z} |\beta^{-1}(\B_i)| = \wt(\UU(\beta_{\bt}(\bl)))
    \end{align*}
    by \eqref{E:wt-quot} and \eqref{E:beta+}.
\end{enumerate}
\end{proof}


\begin{Def}
Let $\AAnobar$ and $\AAbar$ be the subsets of $\Z^{\ell}$ defined by
\begin{align*}
\AAnobar &:= \{ (t_1,\dotsc, t_{\ell}) \in \Z^{\ell} \mid t_1 \leq t_2 \leq \dotsb \leq t_{\ell} \leq t_1 + e-1 \}; \\
\AAbar &:= \{ (t_1,\dotsc, t_{\ell}) \in \Z^{\ell} \mid t_1 \leq t_2 \leq \dotsb \leq t_{\ell} \leq t_1 + e \}.
\end{align*}
\end{Def}

\begin{Def}
For $a, b \in \Z$ with $0 \leq b \leq \ell-1$, define
$$
\bv(a,b) := (a,\dotsc,a, \overbrace{a+1,\dotsc, a+1}^{b \text{ times}}) \in \Z^{\ell}.
$$
\end{Def}

We note that $(\rho\Be_{\ell})^{a\ell + b} = \rho^b \bv(a,b)$, where $\rho = (1,2,\dotsc, \ell) \in \mathfrak{S}_{\ell}$.

\begin{prop} \label{P:AAbar}
Let $\bt = (t_1,\dotsc, t_{\ell}) \in \Z^{\ell}$, and let $\rho = (1,2,\dotsc, \ell) \in \sym_{\ell}$.
\begin{enumerate}
\item   Then
$\bt \in \AAbar$ if and only if $\bt^{\rho\Be_{\ell}} \in \AAbar$.
\item If $\bt \in \AAnobar$ and $\bt^{\sigma\bv} \in \AAbar$ for some $\sigma \in \mathfrak{S}_{\ell}$ and $\bv \in \Z^{\ell}$, then there exist $a,b \in \Z$ with $0 \leq b \leq \ell -1$ such that:
\begin{itemize}
\item $\bv = \bv(a,b)$;
\item $\sigma\rho^{-b} \in \left< \cs_k \mid t_k = t_{k+1} \right>.$
\end{itemize}
\end{enumerate}
\end{prop}

\begin{proof} \hfill
\begin{enumerate}
\item
Note first that $\bt^{\rho\Be_{\ell}} = (t_2, \dotsc, t_{\ell}, t_1 + e)$.  We have
\begin{align*}
\bt \in \AAbar \Leftrightarrow t_1 \leq t_2 \leq \dotsb \leq t_{\ell} \leq t_1 + e
&\Leftrightarrow t_2 \leq \dotsb \leq t_{\ell} \leq t_1 + e \leq t_2 + e
\Leftrightarrow \bt^{\rho\Be_{\ell}} \in \AAbar.
\end{align*}

\item
Let $\bv = (v_1,\dotsc,v_{\ell})$ and let $\bu = (u_1,\dotsc, u_{\ell}) = \bt^{\sigma\bv}$.
Let $J_{\bv} = \{ j \mid v_j \ne v_1 \}$.
For each $j \in J_{\bv}$ (so $v_j \ne v_1$), since $\bu \in \AAbar$, we have
$$
[0,\, e] \ni u_{j}-u_1 = (t_{\sigma(j)} + v_je) - (t_{\sigma(1)} + v_1 e) =  (v_j-v_1)e + (t_{\sigma(j)} - t_{\sigma(1)}) ,
$$
forcing $v_j - v_1 =1$ (and $t_{\sigma(1)} \geq t_{\sigma(j)}$) since $|t_{\sigma(j)} - t_{\sigma(1)}| \leq e-1$.
In addition, if there exists $k$ such that $j < k \leq \ell$ and $k \notin J_{\bv}$, then
since $|t_{\sigma(k)} - t_{\sigma(j)}| \leq e-1$, we have
$$0 \leq u_{k} - u_j = (t_{\sigma(k)} + v_1e) - (t_{\sigma(j)} + (v_1+1)e) \leq (e-1) - e \leq -1,$$
a contradiction.
Thus, $\bv = \bv(v_1,|J_{\bv}|)$, and so
$$
\bt^{\sigma\bv} = \bt^{\sigma\bv(v_1, |J_{\bv}|)}
= \bt^{\sigma\rho^{-|J_{\bv}|}(\rho\Be_{\ell})^{v_1 \ell + |J_{\bv}|}}.
$$
Consequently,
$$
\bt^{\sigma\rho^{-|J_{\bv}|}} = (\bt^{\sigma\bv})^{(\rho\Be_{\ell})^{-(v_1 \ell + |J_{\bv}|)}} \in \AAbar
$$
by part (1).  Hence, $\bt^{\sigma\rho^{-|J_{\bv}|}} = \bt$ and
$\sigma\rho^{-|J_{\bv}|} \in \left< \cs_k \mid t_k = t_{k+1} \right>$ by Lemma \ref{L:sym}.
\end{enumerate}
\end{proof}

We are now ready to prove the first main theorem of this paper.

\begin{thm} \label{T:core and weight}
Let $\bl  \in \PP^{\ell}$ and $\bt \in \Z^{\ell}$, and let $(\bm;\bu) \in (\bl;\bt)^{\EW_{\ell}}$.  Then:
\begin{enumerate}
\item $\core(\bm;\bu) = \core(\bl;\bt).$

\item $\wt(\bm;\bu) = \min (\wt((\bl;\bt)^{\EW_{\ell}}))$
if and only if $\bu \in \AAbar$.
\end{enumerate}
\end{thm}

\begin{proof}
Part (1) follows immediately from Proposition \ref{P:Uglov-partition}(2).

For part (2), we note first that if $\wt(\bm;\bu) = \min (\wt((\bl;\bt)^{\EW_{\ell}}))$, then $\bu \in \AAbar$ by Proposition \ref{P:Uglov-partition}(3,4).
By replacing $(\bl;\bt)$ with $(\bl;\bt)^w$ for suitable $w \in \EW_{\ell}$, we may assume that $0 \leq t_1 \leq \dotsb \leq t_{\ell} \leq e-1$, where $\bt = (t_1,\dotsc, t_{\ell})$, so that $\bt \in \AAnobar$.
To complete the proof, we need to show that $\wt(\bm;\bu) = \wt(\bl;\bt)$ whenever $\bu \in \AAbar$.

Let $\bv = (v_1,\dotsc, v_{\ell}) \in \Z^{\ell}$ and $\sigma \in \mathfrak{S}_{\ell}$ be such that $(\bm;\bu) = (\bl;\bt)^{\sigma\bv}$, and assume that $\bu \in \AAbar$.
By Proposition \ref{P:AAbar}(2), there exist $a,b \in \Z$ with $0 \leq b \leq \ell -1$ such that $\bv = \bv(a,b)$ and $\sigma\rho^{-b} \in \left< \cs_k \mid t_k = t_{k+1} \right>$.
Thus,
\begin{align*}
\wt(\bm;\bu) &= \wt((\bl;\bt)^{\sigma\bv}) = \wt((\bl;\bt)^{\sigma\bv(a,b)})
= \wt((\bl;\bt)^{\sigma\rho^{-b}(\rho^b \bv(a,b))}) \\
&= \wt(((\bl;\bt)^{\sigma\rho^{-b}})^{(\rho\Be_{\ell})^{a\ell +b}}) =
\wt(\bl;\bt)
\end{align*}
by Proposition \ref{P:Uglov-partition}(1,3).
\end{proof}

\section{Blocks of Ariki-Koike algebras} \label{S:AK}

In this section, we apply the results of the last section to the blocks of Ariki-Koike algebras.  In particular, we prove a version of Nakayama's `Conjecture' for these blocks, as well as provide a sufficient condition for Scopes equivalence between two such blocks.

From now on, $\FF$ denotes a fixed field of arbitrary characteristic, and we assume that $\FF$ contains a primitive $e$-th root of unity $q$.
Let $\br = (r_1,\dotsc, r_{\ell}) \in \Z^{\ell}$ and let $n \in \Z^+$.
The Ariki-Koike algebra $\HH_n = \HH_{\FF, q, \br}(n)$ is the unital $\FF$-algebra generated by $\{ T_0,\dotsc, T_{n-1} \}$ subject to the following relations:
\begin{alignat*}{2}
(T_0 - q^{r_1})(T_0 - q^{r_2}) \dotsm (T_0 - q^{r_{\ell}}) &= 0; \\
(T_a - q)(T_a+1) &= 0 &\qquad &(1 \leq a \leq n-1 ); \\
T_0T_1T_0T_1 &= T_1T_0T_1T_0; && \\
T_aT_{a+1}T_a &= T_{a+1}T_a T_{a+1} &\qquad &( 1\leq a \leq n-2); \\
T_aT_b &= T_aT_b &\qquad &(|a-b| \geq 2).
\end{alignat*}
It is clear from the definition that $\HH_n$ only depends on the orbit $\br^{\EW_{\ell}}$ and not on $\br$.

It is well known that $\HH_n$ is a cellular algebra, in the sense of Graham and Lehrer \cite{GL}.
The Specht modules $\Sp^{\bl}$ ($\bl \in \PP^{\ell}(n)$), constructed by Dipper, James and Mathas \cite{DJM}, are known to be cell modules for $\HH_n$.
We note that, unlike $\HH_n$, these Specht modules actually depend on the order of the $r_j$'s.  

Lyle and Mathas gave the first classification of the blocks of $\HH_n$.  For convenience, if $B$ is a block of $\HH_n$ and $\bl \in \PP^{\ell}(n)$, we say that $\bl$ lies in $B$ if and only if $S^{\bl}$ lies in $B$.

\begin{thm}[{\cite[Theorem 2.11]{Lyle-Mathas}}] \label{T:LM-block}
Two $\ell$-partitions $\bl$ and $\bm$ of $n$ lie in the same block of $\HH_n$ if and only if $\res^{\br}(\bl) = \res^{\br}(\bm)$.
\end{thm}

A node with $e$-residue $i$ is called an $i$-node, and so we have another equivalent classification of blocks of $\HH_n$: $\bl$ and $\bm$ lie in the same block if and only if $[\bl]$ and $[\bm]$ has the same number of $i$-nodes for all $i \in \Z/e\Z$. More formally, let $C_i(\bl)$ denote the number of $i$-nodes in $[\bl]$; then $\bl$ and $\bm$ lie in the same block if and only if $C_i(\bl) = C_i(\bm)$ for all $i \in \Z/e\Z$.

Using the ingenious invention of the $e$-abacus by James, it is well known that when $\ell =1$ so that $\bl = (\lambda)$, $\bm = (\mu)$ and $\br = (r)$, we have $\res^{\br}(\bl) = \res^{\br}(\bm)$ if and only if $\core(\lambda) = \core(\mu)$ and $\wt(\lambda) = \wt(\mu)$.
The first main result of this section is a higher analogue of this statement.

\begin{thm} \label{T:core&weight}
Let $\bl, \bm \in \PP^{\ell}$.  Then $\res^{\br}(\bl) = \res^{\br}(\bm)$ if and only if $\core(\bl;\br) = \core(\bm;\br)$ and $\min(\wt((\bl;\br)^{\EW_{\ell}})) = \min(\wt((\bm;\br)^{\EW_{\ell}}))$.
\end{thm}

\begin{proof}
Let $w \in \EW^{\ell}$ be such that $\br^w \in \AAbar$, and let $r = \sum_{j=1}^{\ell} r_j$.
Firstly, $\res^{\br}(\bnu) = \res^{\br^w}(\bnu^w)$ for any $\bnu \in \PP^{\ell}$ by \eqref{E:res-hub}.
Next, $\res^{\br^w}(\bl^w) = \res^{\br^w}(\bm^w)$ if and only if $\res^r(\UU(\bl^w;\br^w)) = \res^r(\UU(\bm^w;\br^w))$ by \cite[Corollary 2.27]{JL-cores}.
Together with the fact that $\res^{r}(\lambda) = \res^{r}(\mu)$ if and only if $\core(\lambda) = \core(\mu)$ and $\wt(\lambda) = \wt(\mu)$, we have
\begin{align*}
\res^{\br}(\bl) = \res^{\br}(\bm)
&\Leftrightarrow \res^{\br^w}(\bl^w) = \res^{\br^w}(\bm^w) \\
&\Leftrightarrow \res^{r}(\UU(\bl^w;\br^w)) = \res^{r}(\UU(\bm^w;\br^w)) \\
&\Leftrightarrow \core(\UU(\bl^w;\br^w)) = \core(\UU(\bm^w;\br^w)) \text{ and } \wt(\UU(\bl^w;\br^w)) = \wt(\UU(\bm^w;\br^w)) \\
&\Leftrightarrow \core(\bl^w;\br^w) = \core(\bm^w;\br^w) \text{ and } \wt(\bl^w;\br^w) = \wt(\bm^w;\br^w) \\
&\Leftrightarrow \core(\bl;\br) = \core(\bm;\br) \text{ and } \min(\wt((\bl;\br)^{\EW_{\ell}})) = \min(\wt((\bm;\br)^{\EW_{\ell}}))
\end{align*}
by Theorems \ref{T:core and weight}.
\end{proof}

In view of Theorem \ref{T:core&weight}, we make the following definition:

\begin{Def}
Let $\bl \in \PP^{\ell}$.  Define
$$
\coreH(\bl) := \core(\bl;\br) \qquad \text{and} \qquad \wtH(\bl) := \min(\wt((\bl;\br)^{\EW_{\ell}})).
$$
\end{Def}

With the above definition, we obtain immediately from Theorems \ref{T:LM-block} and \ref{T:core&weight} the following corollary, which can be considered as a higher analogue of Nakayama's `Conjecture' for blocks of symmetric groups and Iwahori-Hecke algebras of type $A$ (see \cite[6.1.21]{JK}).

\begin{cor} \label{C:Naka}
Let $\bl,\bm \in \PP^{\ell}$.  Then $\bl$ and $\bm$ lie in the same block of an Ariki-Koike algebra if and only if $\coreH(\bl) = \coreH(\bm)$ and $\wtH(\bl) = \wtH(\bm)$.
\end{cor}

\begin{rem}
In \cite{Fayers-weight}, Fayers defined $e$-weights for multipartitions, which we shall denote as $\mathsf{wt}_F(-)$ here to avoid confusion with our definitions in this paper.  It follows from \cite[Theorem 3.4]{JL-cores} and Theorem \ref{T:core and weight}(2) that his definition $\mathsf{wt}_F(-)$ agrees with our definition $\wtH(-)$, i.e.\ $\mathsf{wt}_F(\bl) = \wtH(\bl)$ for all $\bl \in \PP^{\ell}$.
\end{rem}

\begin{Def}
By Corollary \ref{C:Naka}, given a block $B$ of $\HH_n$, we may define, without ambiguity, its $e$-core and its $e$-weight, denoted $\core(B)$ and $\wt(B)$ respectively, as
\begin{align*}
\core(B) := \coreH(\bl) \qquad \text{and} \qquad \wt(B) &:= \wtH(\bl)
\end{align*}
for any $\bl \in \PP^{\ell}$ lying in $B$.
\end{Def}

It is well known that the affine Weyl group $\AW_e$ acts naturally on the blocks of Ariki-Koike algebras. We now give another perspective from Nakayama's `Conjecture' of Ariki-Koike algebras (Corollary \ref{C:Naka}).
Recall the left action $\DDot{\bt}$ (for $\bt \in \Z^{\ell}$) and $\DDot{}$ of $\AW_e$ on $\PP^{\ell}$ and $\PP^{\ell} \times \Z^{\ell}$ respectively, as detailed in the last two paragraphs of Subsection \ref{SS:Weyl group action}.
Let $r  = \sum_{j=1}^{\ell} r_j$.
For $\bl \in \PP^{\ell}$, we have
\begin{align*}
\coreH(\cs_i \DDot{\br} \bl)
&= \beta^{-1}( \core(\UU(\beta_{\br}(\cs_i \DDot{\br} \bl))))
= \beta^{-1} (\core( \UU(\cs_i \DDot{1} \beta_{\br}(\bl)))) \\
&= \beta^{-1}( \core( \cs_i \DDot{\ell} \UU(\beta_{\br}(\bl))))
= \beta^{-1} (\cs_i \DDot{\ell} \core( \UU(\beta_{\br}(\bl))))
= \beta^{-1} (\cs_i \DDot{\ell} \beta_r(\coreH(\bl)))
\end{align*}
for all $i \in \Z/e\Z$, where the first two equalities in the second line follow from \eqref{E:cs_i-U} and \eqref{E:cs_i-core} respectively.
On the other hand, since the left action $\DDot{}$ of $\AW_e$ on $\PP^{\ell} \times \Z^{\ell}$ commutes with the right action of $\EW_{\ell}$,
we have
$$
\cs_i \DDot{} (\bl;\br)^{\EW_{\ell}} = (\cs_i \DDot{} (\bl;\br))^{\EW_{\ell}} = (\cs_i \DDot{\br} \bl; \br)^{\EW_{\ell}}.$$
Furthermore,
\begin{align*}
\wt(\cs_i \DDot{\bt} \bm;\bt)
&= \wt(\UU(\beta_{\bt}(\cs_i \DDot{\bt} \bm)))
= \wt(\UU(\cs_i \DDot{1} \beta_{\bt}(\bm))) \\
&= \wt(\cs_i \DDot{\ell} \UU(\beta_{\bt}(\bm)))
= \wt(\UU(\beta_{\bt}(\bm))) = \wt(\bm;\bt)
\end{align*}
for all $(\bm;\bt) \in \PP^{\ell} \times \Z^{\ell}$ by \eqref{E:cs_i-wt}.
Thus,
\begin{align*}
\wtH(\cs_i \DDot{\br} \bl)
&= \min( \wt((\cs_i \DDot{\br} \bl;\br)^{\EW_{\ell}})) \\
&= \min( \wt(\cs_i \DDot{} (\bl;\br)^{\EW_{\ell}} )) \\
&= \min\{ \wt(\cs_i \DDot{\bt} \bm;\bt) \mid (\bm;\bt) \in (\bl;\br)^{\EW_{\ell}} \} \\
&= \min\{ \wt(\bm;\bt) \mid (\bm;\bt) \in (\bl;\br)^{\EW_{\ell}} \}
= \wtH(\bl).
\end{align*}

The upshot of this is that each $\cs_i$ sends the $\ell$-partitions lying in a block $B$ to the $\ell$-partitions lying in the block $\cs_i \DDot{} B$, where
\begin{align*}
\core(\cs_i \DDot{} B) &= \beta^{-1} (\cs_i \DDot{\ell} \beta_r(\core(B))); \\
\wt(\cs_i \DDot{} B) &= \wt(B).
\end{align*}


Since $\cs_i \DDot{\br} \bl$ is obtained from $\bl$ by removing all its removable $i$-nodes and adding all its addable $i$-nodes, both relative to $\br$, we see $\cs_i \DDot{\br} \bl \in \PP^{\ell}(n - \hub_i^{\br}(\bl))$ if $\bl \in \PP^{\ell}(n)$ (see Subsection \ref{SS:multipartitions} for the definition of $\hub^{\bt}_i(\bl)$ for $(\bl;\bt) \in \PP^{\ell} \times \Z^{\ell}$).
Thus $\cs_i \DDot{} B$
is a block of $\HH_{n-\hub_i(B)}$ when $B$ is a block of $\HH_n$, where
$\hub_i(B) = \sum_{j=1}^{\ell} \hub^{\br}_i(\bl)$ for any $\ell$-partition $\bl = (\lambda^{(1)},\dotsc, \lambda^{(\ell)})$ lying in $B$.
Note that $\cs_i \DDot{} B = B$ if and only if $\hub_i(B) = 0$.

\begin{rem}
Fayers \cite{Fayers-weight} called the $e$-tuple $(\delta_0(B),\dotsc, \delta_{e-1}(B))$ the {\em ($e$-)hub of $B$}, and showed that $B$ is completely determined by its ($e$-)hub and its ($e$-)weight.
\end{rem}

\begin{Def}
Let $B$ be a block of $\HH_n$ and let $i \in \Z/e\Z$.  If $B \ne \cs_i \DDot{} B$, we say that the blocks $B$ and $\cs_i \DDot{} B$ {\em form a $[w:k]$-pair}, where
\begin{align*}
w &= \wt(B)\ (= \wt(\cs_i \DDot{} B)); \\
k &= |\hub_i(B)|\ (= |\hub_i(\cs_i \DDot{} B)|).
\end{align*}
\end{Def}





The following theorem is well known among experts.

\begin{thm}[see {\cite[Theorem 6.4]{CR}}]
Let $B$ be a block of $\HH_n$ and $i \in \Z/e\Z$.  Suppose that for all $\ell$-partitions $\bl$ lying in $B$, $\bl$ has no addable $i$-node relative to $\br$.  Then $B$ and $\cs_i \DDot{} B$ are Morita equivalent, with the `divided powers' of $i$-restriction and $i$-induction functors giving the equivalence between their module categories.  Under this equivalence, the Specht module $S^{\bl}$ lying in $B$ corresponds to $S^{\cs_i \DDot{\br} \bl}$.
\end{thm}

We note that an $\ell$-partition $\bl$ has no addable $i$-node relative to $\br$ if and only if $\cs_i \DDot{\br} \bl$ has no removable $i$-node relative to $\br$.

Let $B$ be a block of $\HH_n$ and $i \in \Z/e\Z$.  Following \cite{Lyle-RoCK, Webster}, we say that $B$ and $\cs_i \DDot{} B$ are {\em Scopes equivalent} if all $\ell$-partitions lying in $B$ have no addable $i$-node or all $\ell$-partitions lying in $B$ have no removable $i$-node.
Clearly, Scopes equivalence is symmetric by definition, and we further extend it to an equivalence relation on all blocks by taking its reflexive and transitive closure.


We have the following lemma by Scopes:

\begin{lem}[{Proof of \cite[Lemma 2.1]{Scopes}}] \label{L:Scopes}
Let $\B$ be a $\beta$-set. If $\wt(\B) \leq \hub_i(\B)$, then there does not exist $x \in \B$ such that $x\equiv_e i-1$ and $x + 1 \notin \B$.
\end{lem}

\begin{thm} \label{T:Scopes}
Let $B$ and $C$ be blocks forming a $[w:k]$-pair, with $k \geq w$.  Then $B$ and $C$ are Scopes equivalent.
\end{thm}

\begin{proof}
Let $C = \cs_i \DDot{} B$.  By interchanging the roles of $B$ and $C$ if necessary, we may assume that $k = \hub_i(B)$.
Let $\bl = (\lambda^{(1)}, \dotsc, \lambda^{(\ell)})$ be an $\ell$-partition lying in $B$, and let $(\bm;\bt) \in (\bl;\br)^{\EW_{\ell}}$ be such that $\wt(\bm;\bt) = \min(\wt((\bl;\br)^{\EW_{\ell}}))$.
Note that if $(\bm;\bt) = (\bl;\br)^{\sigma\bu}$ ($\sigma \in \sym_{\ell}$, $\bu \in \Z^{\ell}$), there is an obvious correspondence between the nodes $(a,b,i) \in [\bm]$ and those of $[\bl]$ as dictated by $\sigma$, and that under this correspondence the $e$-residues relative to $\bt$ of the former are exactly the same as the $e$-residues relative to $\br$ of the latter (since $\bt = (\br^{\sigma})^{\bu} \equiv_e \br^{\sigma}$). Consequently, $\bl$ has no addable $i$-node relative to $\br$ if and only if $\bm$ has no addable $i$-node relative to $\bt$.

Let $\bt' = \bt + \delta_{i0}\bone$ and $i' = i + \delta_{i0}$, where $\bone = \sum_{j=1}^{\ell} \Be_j \in \Z^{\ell}$.
Then $i' >0$, and an $i$-node relative to $\bt$ is precisely an $i'$-node relative to $\bt'$.  In particular, $\bm$ has no addable $i$-node relative to $\bt$ if and only if it has no addable $i'$-nodes relative to $\bt'$.

By Proposition \ref{P:Uglov-partition}(5), we have
$$
\wt(\UU(\beta_{\bt'}(\bm))) = \wt(\bm;\bt') = \wt(\bm;\bt) = w.$$
In addition,
$$
\hub_{i'} (\UU(\beta_{\bt'}(\bm))) = \sum_{j=1}^{\ell} \hub_{i'} (\beta_{t'_j}(\mu^{(j)})) = \sum_{j=1}^{\ell} \hub_{i}(\beta_{r_j}(\lambda^{(j)})) = \hub^{\br}_i(\bl) = \hub_i(B) = k
$$
by Lemma \ref{L:}(3).
Thus, there does not exist $x \in \UU(\beta_{\bt'}(\bm))$ with $x \equiv_e i'-1$ such that $x + 1 \notin \UU(\beta_{\bt'}(\bm))$ by Lemma \ref{L:Scopes}.

Now, if $\bm$ has an addable $i'$-node relative to $\bt'$,
then there exists $a \in \beta_{t'_j}(\mu^{(j)})$ for some $j \in \{1,\dotsc, \ell\}$ such that $a \equiv_e i'-1$ and $a+1 \notin \beta_{t'_j}(\mu^{(j)})$,
and so
$$
\uu_j(a) \in \uu_j(\beta_{t'_j}(\mu^{(j)})) \subseteq \UU(\beta_{\bt'}(\bm))
$$
while
$$
\uu_j(a) + 1 = \uu_j(a+1) \notin \UU(\beta_{\bt'}(\bm)),
$$
giving us a contradiction, since $\uu_j(a) \equiv_e a$ by Lemma \ref{L:Uglov map}(1).
Consequently $\bm$ has no addable $i'$-nodes relative to $\bt'$ and hence $\bl$ has no addable $i$-node relative to $\br$ as desired.
\end{proof}

The converse of Theorem \ref{T:Scopes} is true for $\ell =1$, but not true for $\ell \geq 2$; i.e.\ it is possible for a $[w:k]$-pair $B$ and $C$ to be Scopes equivalent even when $k < w$ if $\ell \geq 2$:

\begin{eg} \label{E:counter-eg}
Let $\ell = 2$, $e = 5$, $\br = (2,2)$, and let $B$ be a block of $\HH_5$ with $e$-weight $2$ and empty $e$-core.  Then the $2$-partitions of $5$ lying in $B$ are
$$
((3,2),\varnothing), (\varnothing, (3,2)), ((3,1),(1)), ((1),(3,1)), ((3), (1,1)), ((1,1), (3)).
$$
It can be easily verified that each of these partitions has exactly one removable $0$-node but no addable $0$-node relative to $\br$, so that $B$ and $\cs_0 \DDot{} B$ form a $[2:1]$-pair and are Scopes equivalent.
\end{eg}

\begin{rem}
Dell'Arciprete \cite{Dell'Arciprete} has looked at Scopes equivalences between blocks of Ariki-Koike algebras, and provided a sufficient condition for $B$ and $\cs_i \DDot{} B$ to be Scopes equivalent.  Her sufficient condition is actually satisfied by the blocks in Example \ref{E:counter-eg}.

On the other hand, there are also examples of blocks satisfying the conditions of Theorem \ref{T:Scopes} (and hence are Scopes equivalent) that do not satisfy her condition.  One such example is the block $B$ containing the $4$-partition $((2),(3), (1^3), (1^2))$ with $4$-charge $(1,1,3,3)$ (where $e \geq 4$) and $\cs_2 \DDot{} B$.
\end{rem}

We end this paper with the following remark:
\begin{rem}
We briefly treat the case where $e = \infty$ (i.e.\ $q \in \FF^{\times}$ not being a root of unity) here.  It is not difficult to see that all results in this paper continue to hold in this case as long as $e$ is taken to be sufficiently large.  For example, two Specht modules $\Sp^{\bl}$ and $\Sp^{\bm}$ lie in the same block of $\HH_{\FF,q,\br}(n)$ (where $q$ is not a root of unity) if and only if $\mathsf{core}_{e'}(\bl;\br) = \mathsf{core}_{e'}(\bm;\br)$ and $\min(\mathsf{wt}_{e'}((\bl;\br)^{\EW_{\ell}})) = \min(\mathsf{wt}_{e'}((\bm;\br)^{\EW_{\ell}}))$ for all large enough $e'$.
\end{rem}

%


\section*{Data availability}
We do not analyse or generate any datasets, because our work
proceeds within a theoretical and mathematical approach. One
can obtain the relevant materials from the references below.

\section*{Competing interests}
The authors declare no competing interests.

\begin{footnotesize}

\end{footnotesize}

\end{document}